\documentclass[12pt]{article}
\usepackage{amssymb, url, amsmath, amsthm, amscd}

\newtheorem{theorem}{\noindent Theorem}
\newtheorem{lemma}{\noindent Lemma}
\newtheorem{definition}{\noindent Definition}
\newtheorem{corollary}{\noindent Corollary}

\newtheorem{statement}{\noindent Proposition}
\newtheorem{remark}{\noindent Remark}
\newtheorem{problem}{\noindent PROBLEM}
\newtheorem{definition-theorem}{\noindent Definition-Theorem}

\urldef\vershik\url{vershik@pdmi.ras.ru}
\author{A.~M.~Vershik\thanks{
St.~Petersburg Department of Steklov Institute of Mathematics and
Max Plank Institute Bonn. E-mail: \vershik. Partially supported by the
RFBR-grant 11-01-00677-a; RFBR-grant 11-01-12092 OFI-M }}

\title{\textbf{TOTALLY NONFREE ACTIONS AND THE INFINITE SYMMETRIC GROUP}}

\begin{document}
\maketitle

\rightline{\it \textbf{To the memory of my beloved Lyotya}}

\bigskip

\centerline{\textbf{ABSTRACT}}

We consider the totally nonfree (TNF) action of a groups and the corresponding
adjoint invariant (AD) measures on the lattices of the subgroups of the given group.
 The main result is the description of all adjoint-invariant and TNF-measures on the lattice
of subgroups of the infinite symmetric group $S_{\Bbb N}$. The problem
is closely related to the theory of characters and factor representations of groups.

\section{INTRODUCTION}

   The main result of this paper is a precise description of all
   adjoint-invariant ergodic probability Borel measures on the lattice
   of all subgroups of the infinite symmetric group $S_{\Bbb N}$. The
   reason why problems of this type are of importance can be
   briefly formulated as follows: the adjoint action of a group on
   the lattice of subgroups with an adjoint-invariant probability
   measure produces, in a certain way,
   a nontrivial character of the group and, consequently, determines a
   special factor representation of the group.

   In the case of the infinite symmetric group, it turns out that this
   method yields, {\it all characters of $S_{\Bbb N}$}
   This phenomenon was first observed in \cite{VK}, as a
   particular fact related to a certain model of factor representations
   of the infinite symmetric group. The list of characters itself is
   well-known; E.Thoma posed and gave the first solution of the problem (\cite{Th}),
   the new proof of this theorem, which used the ergodic approach and approximation
   suggested in \cite{VK1}. This proof was based on the ideas of the dynamical approach,
   and approximation of the characters of infinite group with the characters of finite symmetric
    group. The same dynamical approach, namely, so-called grouppoid model, gives the realization of
   corresponding factor-representations of type $II_1$. But in this paper for us more important, that
   the value of an indecomposable character at a given element of the group
   is equal, (up to a certain factor), to the measure of the set of
   fixed points of this element for some special action.
   The most important thing is that precise link between Thoma's parameters of the characters
   and parameters of the measure both are the same.

   But what kind of the actions can appear in this construction? This is what we define in this paper:
   it turns out that these are so-called {\it totally nonfree} ({\it
   TNF}) actions, so it is important to describe such actions for a given group.

   In this paper, we go in the opposite direction: we start with definition and
   studying of the class of TNF actions of a group. From the point of view of ergodic
   theory, this kind of actions is of great interest, and, as far as
   we know, it has not yet been systematically investigated.
   Due to the lack of space, we decided to separate the discussion of
   the link between the questions considered here and the
   theory of representations, characters,
   and factors; these topics will be treated in another article.

  In the first part of paper (the second section), we introduce the main notions and fix
  definitions related to nonfree and totally nonfree actions. We
  develop a systematic approach to nonfree actions. Although we are
  mainly interested in totally nonfree actions, we also consider
  intermediate cases and the reduction of a general action to a TNF
  action.
  The main open problem that arises in this context concerns the {\it
  existence and the list of TNF-actions for a given group}. We use the
  language of the {\it lattice of subgroups and the adjoint action} of
  the original group on this lattice. The previous question is equivalent to that of the
  existence or nonexistence of continuous adjoint-invariant measures.
  The universal example of a TNF action is the adjoint action of the
  group on the lattice of its subgroups with a TNF measure, i.e.,
   a measure concentrated on the set of self-normalizing subgroups (=
  subgroups coinciding with their normalizers). An important fact
  asserts that this measure is a complete metric invariant of the
  action.  In general, there are other adjoint-invariant measures that
  are not TNF; for example so called RTNF-measures which also produced
  TNF action. We describe so called canonical skew-product of any action
  and sequence (which may be infinite of even transfinite) of reduced
  actions.

  All these facts heavily depend on the properties of the group. We consider
  here only countable groups  Of course, for the problem to be
  interesting, the group should have a continuum of subgroups. It is
  clear that for many groups, f.e for groups close to
  abelean ones, such a measure does not exist. But for
  some ``large'' groups,  we have many TNF measures
  (or many TNF actions), which, fortunately, can be listed up to isomorphism (in
  contrast to the usual situation in ergodic theory). It is natural to consider adjoint measures
  on the lattice of subgroups as ``random subgroups''; the notion of a random subgroup was considered
  in \cite{AGV,IRS,LB} but with the different kind of applications.
  The author believes that this question is also of interest within ergodic theory itself.

    In the second part of the paper (the third section), we study the case of the infinite symmetric
    group. We use a fundamental classical fact about its
    subgroups; namely, the infinite symmetric group has no primitive subgroups except the alternating
  subgroup and the whole group itself. This follows from a classical theorem
  due to C.~Jordan (H.~Wielandt was perhaps the first to observe this, see \cite{DM}).
  We successively exclude from consideration all other subgroups
  that cannot lie in the support of an adjoint-invariant measure and
  reduce the question to a de Finetti-like problem and to Kingman's theorem about random partitions of the naturals.
  The description of adjoint-invariant measures on the group $S_{\Bbb N}$ relies on the {\it new important generalization of the
  classical notion of Young subgroup, - namely, a signed Young subgroup; it is natural to understand a
  random signed Young subgroup exactly as a random subgroup whose
  distribution is an adjoint-invariant measure on $L(S_{\Bbb N})$}. The list of parameters $\alpha$ for these
  measures is exactly the list of Thoma's parameters. We briefly
  compare our formula with that of Thoma at the end of the paper; in a
  sense, our list of adjoint-invariant measures can be regarded
  as an explanation for the list of characters. We will return to this question elsewhere.

   The direct proof of the  TNF-measures for infinite symmetric group with ergodic method
   perhaps gives us a new proof of the list of the characters of this group.
    The conception presented here partially is based on the paper \cite{VK}, but the general approach and link to the theory of characters
    is new one, it was proclaimed firstly in the authors' talk in the Henri Poincare institute \cite{IHP}. The short announcement
    of the approach can be found in \cite{TNF}.

     Some papers on the different topics are tightly related to out topic: the papers by R.~Grigorchuk and his colleagues \cite{Grigor,Nagni} contain examples of TNF actions of groups acting on trees. Also, papers by L.~Bowen \cite{LB}, found such examples of nontrivial AD-measures for the free group. The papers \cite{AGV,IRS} devoted to IRS = invariant random subgroups or AD-measures on the lattices in our terms but the goals are different.

    As it known for author the explicit description of the list of all AD-measures and TNF-actions for the group $S_{\Bbb N}$
    which we present here, is the first result of his type. Perhaps, this methods can also be applied to other groups similar to
    $S_{\Bbb N}$, such as the group of rational interval exchange, $U(\infty)$ (the infinite unitary group), $GL(F_q,\infty)$, etc.
    It turns out that our answer is even more tightly connected to group-theoretic structure and to the theory of characters,
    than it can be assumed before; we will apply it to the theory of characters and factor-representations of $S_{\Bbb N}$ and
    other groups in a subsequent paper.

    Professors  M.Abert, L.Bowen, Y.Glasner, R.Grigorchuk, Y.Guivarch, N.Gordeev, W.Knapp, T.Nagnibeda-Smirnova, G.Olshansky, M.Zischang
     gave me the important references on the subject. I am grateful to Natalia Tsilevich for her help in editing of this article.

\section{\textbf{MAIN DEFINITIONS. TOTALLY NONFREE ACTIONS}}

\subsection{FIXED POINTS, STABILIZERS, AND SUB-$\sigma$-FIELDS}

Let $(X,{\frak A},\mu)$ be a Lebesgue space with a probability measure
$\mu$ defined on a $\sigma$-field ${\frak A}$ of classes
of $\bmod 0$ coinciding measurable sets, and let a countable group $G$
act on this space by $\mu$-preserving transformations. We will consider
only {\it effective actions}, which means that only the identity $e \in G$ of the
group acts as the identity transformation $\bmod 0$. Because of that we denote by the same
letter an element of the group ($g \in G$) and the corresponding automorphisms ($g:x\mapsto gx$)
 of the space $(X,\mu)$.

For each element $g\in G$, we define a measurable set
$X_g$ called the {\it set of fixed points} of $g$:
$$X_g=\{x\in X, \quad gx=x\}.$$
 Consider the map $$\Phi: G \to {\frak A}; \quad \Phi:g
\mapsto X_g.$$

\begin{definition} The fixed point $\sigma$-field corresponding to the
action of $G$ under consideration is the sub-$\sigma$-field ${\frak A}_G$ of the
$\sigma$-field $\frak A$ generated by the family of all sets $X_g$,  $g\in G$.
\end{definition}

The sets $X_g$ are well defined for arbitrary actions of countable
groups and, more generally, for {\it pointwise, or measurable actions} of arbitrary
groups.\footnote{Recall that an action of a group $G$ is called pointwise
(or measurable) if there is a measurable set of full measure on which the action of $G$
is defined; an action is called continuous
(in Rokhlin's terminology; the other term is $\bmod 0$-action)
if a homomorphism $G\rightarrow\operatorname{Aut}(X,\mu)$ is defined,
where  $\operatorname{Aut}(X,\mu)$ is the group of all classes of
measure-preserving
transformations of $(X,\mu)$. For countable groups, as well
for locally compact groups, these two
notions are equivalent.}
It is worth mentioning that the above definition of the  $\sigma$-field ${\frak A}_G$
applies to {\it continuous actions of arbitrary groups}, since
 the set of fixed points for a given automorphism is well defined
 with respect to $\bmod 0$: if $g_1=g_2 \bmod 0$ (as the automorphisms of the space $(X,\mu)$),
 then $X_{g_1}=X_{g_2}\bmod 0$.

\begin{remark} {\rm An action of a group $G$ is called free
if $\mu X_g=0 $ for all $g \ne{\rm Id}$, $g \in G$, or, in short, if the
$\sigma$-field ${\frak A}_G$ is trivial (the trivial $\sigma$-field
will be denoted by $\frak N$).}
\end{remark}

For pointwise actions, we can define the notion of the {\it
stationary subgroup}, or the {\it stabilizer},
of a point $x \in X$:
$$G_x=\{g \in G: gx=x\}.$$
It is clear that if $y=hx$ with $y,x \in X_0$, $h\in G$,
then $G_y=h^{-1}G_xh$. In general, this notion is not well defined for
uncountable groups; more exactly, it can be defined only if one can
introduce the notion
of the orbit partition.

Now we are going to define another sub-$\sigma$-field of the $\sigma$-field $\frak A$ in the space $X$.
We start with the following definition.

\begin{definition} Consider the partition $\xi_G$ of the space $X$
into the classes of points having the same stabilizer.
We call it the iso-stable partition of the triple $(X,G,\mu)$.
\end{definition}

 The iso-stable partition $\xi_G$ is measurable, because it is the
 limit, over an increasing sequence of finite subsets $K_n\subset G$,
 $\bigcup_n K_n=G$, of measurable partitions $\xi_G^{K_n}$: $\xi_G=\lim_n \xi_G^{K_n}$, where two points $x,y\in X$ belong to the same block of $\xi_G^{K_n}$
 if and only if  $K\subset G_x$, $K\subset G_y$. The partition $\xi_G$
 is obviously $G$-invariant, because an element of
$\xi_G$ consists of all points that have the same
stabilizers.

\begin{definition}
Let ${\frak A}^G$ be
the sub-$\sigma$-field of $\frak A$ that consists of
all sets measurable with respect to the iso-stable partition $\xi_G$.
In the quotient space $X/\xi_G$, we have a well-defined action of the group $G$
with invariant quotient (projection) measure $\mu_{\xi_G}$;
the action of $G$ on  $(X/\xi_G,\mu_{\xi_G})$ will be called the reduced action.
\end{definition}

\begin{statement}
Assume that there is a pointwise action of a group $G$ on a space $(X,\mu)$ with an
invariant measure $\mu$.
Both sub-$\sigma$-fields ${\frak A}_G$ and ${\frak A}^G$ are
$G$-invariant, and
the following inclusion holds:
$${\frak A}_G\subset {\frak A}^G.$$
 For a countable group $G$, both sub-$\sigma$-fields coincide:
$${\frak A}_G ={\frak A}^G \equiv {\frak A}(G).$$
\end{statement}

\begin{proof}
The first claim is trivial: two points that cannot be separated by
their fixed point sets
have the same stabilizers.
By definition, the $\sigma$-field ${\frak A}_G$ is generated by the family
of sets $X_g$, $g \in G$. But, since the group $G$ is countable,
a basis of the $\sigma$-field ${\frak
A}^G$ consists of the sets
$$Y_K=\cap_{g \in K} X_g,$$
where $K \subset\operatorname{Stab} (x)\subset G$ is an arbitrary
finite set. Thus the
family $X_g$, $g\in G$, generates both $\sigma$-fields in question.
\end{proof}

For continuous groups, the sub-$\sigma$-fields in question do not
coincide in general. For instance, considering the action of the orthogonal group $SO(3)$ on the
projective space $RP_2$, we see that in this case ${\frak A}_G
\subsetneqq {\frak A}^G$. Indeed,
each set of fixed points has zero measure, whence ${\frak A}_G=\frak N$
(where $\frak N$ is the trivial $\sigma$-field),
while ${\frak A}^G=\frak A$ since the set of all rotations
separate the points of $P_2R$.

\subsection{THE LATTICE OF SUBGROUPS AND THE ADJOINT ACTION}

 Denote by $L(G)$ the set of all subgroups of a locally compact group
 $G$ and equip it with the natural weak topology\footnote{A
 neighborhood of a subgroup in the weak
topology is the set of subgroups that have the same intersection
with a given compact subset of $G$. For a discrete group,
$L(G)$ is a subspace of the compact space of all subsets of $G$.}
and the corresponding Borel structure.
For a countable group, the space $L(G)$ equipped with this topology is a compact (Cantor) space. The {\it adjoint action} of the group $G$ on $L(G)$
is defined as follows. Let $g \in G$, $H \in L(G)$; then
$$\operatorname{Ad}(g)H= gHg^{-1}.$$
We will study the dynamical system $(L(G), \operatorname{Ad}(G))$ from
the point of view of ergodic theory; namely, we will consider
$\operatorname{Ad}(G)$-invariant Borel measures. The key problem is
the existence of continuous (nonatomic) invariant measures.

\begin{problem}
For what groups do there exist continuous
$\operatorname{Ad}(G)$-invariant Borel probability measures? Describe
all such measures for a given group.
\end{problem}

 We will solve this problem for the infinite symmetric group.
 Of course, the theory we develop here is interesting for countable
 groups that have uncountably many subgroups.

 It is known (see \cite{LB}) that a non-Abelean free group has a lot
 of such measures, but one has no general description of these measures.
 In \cite{Grigor}, actions of groups on trees and more general graphs
 were considered, and it was verified that these
 actions are TNF.

 A natural point of view on $\operatorname{Ad}(G)$-invariant measures is
 to regard them as {\it random subgroups of $G$};
 more precisely, each $\operatorname{Ad}(G)$-invariant measure
 determines a statistics on the set of subgroups, or a random subgroup.
 The invariance under conjugations is a  natural condition for
 applications. One may refine this condition and consider
 random subgroups with additional properties (e.g., TNF measures, or
 $\operatorname{Ad}(G)$-invariant measures on the set of self-normalizing subgroups, see below).
 In the recent paper \cite{IRS}, a problem related to random
 subgroups arises for a different reason.

  The lattice structure on the space of subgroups $L(G)$ is a
  very popular object of algebraic studies (see, e.g., \cite{Sch});
  we will not use it. It is worth mentioning that an important and
  completely open question concerns the existence of {\it $\sigma$-finite invariant continuous  measures on $L(G)$}.
  As far as we know, ergodic aspects of the natural dynamical system
  $(L(G),\operatorname{Ad}(G),\nu)$, where $\nu$ is an
  $\operatorname{Ad}(G)$-invariant measure,
 has not been seriously studied.

 Let us now connect these dynamical systems
 $(L(G),\operatorname{Ad}(G),\nu)$ with nonfree actions of the group $G$.
Namely, we can identify the quotient space with respect to the iso-stable
partition $\xi_G$ with $L(G)$.

\begin{definition}
Given an action of a group $G$ on a Lebesgue space $(X,\mu)$, consider the map
$$ \Psi:X \rightarrow L(G),\quad  \Psi(x) =  G_x.$$
It is a measurable homomorphism of the triple $(X,G,\mu)$
to the triple $(L(G),\operatorname{Ad}(G),\Psi_*\mu)$,
where $\Psi_*\mu$ is an $\operatorname{Ad}(G)$-invariant Borel measure
on $L(G)$, the image of the measure $\mu$ under $\Psi$:
$$\Psi_*(\mu)(B)=\mu\{x:G_x \in B\subset L(G)\}.$$
We will call $\Psi_*\mu$ the characteristic measure of the action $(X, G,\mu)$.
\end{definition}

From definitions it is clear that $\Psi$ is {\it isomorphism} between the reduced actions of the group $G$
on $(X/\xi_G,\mu_{\xi_G})$  and adjoint action on $(L(G),\Psi_*\mu$.
\begin{statement}
The characteristic measure  $\Psi_*\mu$ is a metric invariant of
measure-preserving actions in the following sense:
if two measure-preserving actions of a countable group $G$ on spaces
$(X^i,\mu^i)$, $i=1,2$, are metrically isomorphic, then the
corresponding measures  $\Psi_*\mu^i$,  $i=1,2$, on $L(G)$ coincide.
\end{statement}

\begin{proof}
It suffices to observe that every isomorphism between two actions of $G$
must send the set of points with a given stabilizer
for one action to the same set for the other action.
\end{proof}

 The map $\Psi$ is nothing else than the factorization of the space
 $(X,\mu)$ with respect to the iso-stable partition $\xi_G$, which
 identifies the quotient space $X/{\xi_G}$ with the image $L(G)$.
 The quotient measure $\mu_{\xi_G}$ tautologically coincides with the characteristic measure $\Psi_*\mu$.

For a free action, $\Psi$ is a constant map and the characteristic
measure is the $\delta$-measure at
 the identity subgroup $\{e\}\in L(G)$. If the action of the group is effective then $\bigcap_x G_x=\{e\}$.

\subsection{TOTALLY NONFREE (TNF) ACTIONS}

\begin{definition-theorem}
A measure-preserving action of a countable group $G$ on a
space $(X, \mu)$ is called totally nonfree (TNF) if one of the
following equivalent conditions holds:

{\rm 1.} The $\sigma$-field ${\frak A}_G$ $(={\frak A}^G={\frak
A}(G))$
generated by all sets of fixed points
coincides with the whole $\sigma$-field $\frak
A$ of all measurable subsets of $X$. Equivalently, the iso-stable partition $\xi_G$ coincides
$\bmod 0$ with the partition into separate points.

{\rm 2.} The map $\Psi:X \rightarrow X/{\xi_G}\simeq L(G)$
is a  $\bmod0$ isomorphism $\bmod 0$ of the action of $G$ on $(X,\mu)$ and adjoint action on $(L(G),\Psi_*\mu$.

If an action is TNF, then we say that its characteristic
measure is a TNF measure on $L(G)$.
\end{definition-theorem}

The equivalence of the above two conditions directly follows from the
definitions of the previous section.
 It is also clear that the definitions are correct with respect to
changing the actions on sets of zero measure.

TNF actions are an opposite extreme to free actions.

{\it The characteristic measure of the ergodic TNF-action is complete metric invariant}
therefore the metric classification of TNF actions of a countable group $G$
reduces to the calculation of the characteristic measures $\Psi_*\mu$ on the lattice $L(G)$.
Thus the classification problem for TNF actions is, in a sense,  smooth (tame),
in contrast to the general isomorphism problem in ergodic theory.

\begin{definition}
 The {\it normalizer} of a subgroup $\Lambda \subset G$ is the subgroup
$N(\Lambda)=\{g\in G: g\Lambda g^{-1}=\Lambda$\}.
A subgroup $H \subset G$ for which $N(H)=H$ is called {\it
self-normalizing}.\footnote{It is more natural to call such a subgroup
{\it abnormal}, or {\it anormal}.}
Denote the set of all self-normalizing subgroups of $G$ by $LN(G)$.
\end{definition}

The following claim is obvious.

\begin{statement}
A transitive action of a group $G$ (the left action of $G$ on a
homogeneous space $G/H$) is totally nonfree
if and only if
the stabilizer (i.e., $H$) is
a self-normalizing subgroup ($N(H)=H$, or $H\in LN(G)$).
\end{statement}

Indeed, any two points belong to the same orbit (for any $x,y\in X$,
there exists $h\in G$ such that $y=hx$);
 if they have the same
stabilizer $G_x$, then $h\in N(G_x)$, where
$N(G_x)=\{h\in G: h^{-1}G_xh=G_x\}$ is the normalizer of $G_x$. Consequently,
either $N(G_x)\ne G_x$ and the action is not TNF, or $N(G_x)= G_x$
and the action is TNF.

A similar situation holds for general actions.

\begin{statement}
{\rm1.} If a  measure-preserving action of a group $G$ on a space $(X,\mu)$ is
TNF, then for almost all $x \in X$ the stabilizers $G_x$
are self-normalizing: $N(G_x)=G_x;\quad \mbox{or}\quad \mu\{x:G_x \in LN(G)\}=1$. In other words, the characteristic
measure of the TNF-action is concentrated on the set $LN(G)$ of self-normalizing subgroups.

{\rm2.} The adjoint action of the group $G$ on the measure space
$(L(G),\nu)$ is TNF if and only if $\nu$-almost all subgroups $H\in L(G)$ have self-normalizing
normalizator: $N(N(H))\equiv N^2(H)=N(H)$. In particular, the adjoint action of the group on the lattice $(L(G),\nu)$ with
an $\operatorname{Ad}(G)$-invariant TNF-measure $\nu$ is a TNF action.
\end{statement}

\begin{proof}
1. Assume that the action is TNF, but at the same time there exists a measurable set $A$ of positive $\mu$-measure such that
the stabilizer $G_x$ of every point  $x\in A$ is not self-normalizing.
Then there exists a point $hx\in A$ with $h\in G$, $h \notin G_x$ such
that $hx\neq x$ but $hG_xh^{-1}=G_x$; consequently, $x$ and $hx$ have
the same stabilizer, which contradicts the TNF property.

2. For adjoint action of the group $G$ on $L(G)$ the stabilizer $G_H=N(H)$, so condition $N(G_H)=G_H$ is equivalent to the condition
 $N(H)=N^2(H)$ for $\nu$-almost all $H$; by the item 1 we have TNF-action.
 \end{proof}
\begin{remark}
1.As we saw the adjoint action on $L(G),\nu)$ could be TNF-action not only for TNF-measures but for $\nu$ with property $\nu\{H:N^2(H)=H\}=1$.
{\it  We will call this measures RTNF-measures}. In other words - TNF adjoint action takes place for measures $\nu$ which are concentrated
on $LN^2(G): \nu(LN^2(G))=1$. We will see that for infinite symmetric group we have the examples of those measures.

2.The condition $\mu\{x:G_x\in LN(G)\}=1$ is only
necessary but not sufficient for the action to be TNF,
because the stabilizers of two points $x,y$ that belong to different
orbits can be the same self-normalizing subgroup: $G_x=G_y \in LN(G)$.
\end{remark}

Now we can specify Problem 1 formulated above and reduce the
description of TNF actions of a group to the following question.

\begin{problem}
Given a group $G$, describe all ergodic continuous (nonatomic) probability TNF (correspondingly RTNF) measures on the lattice $L(G)$ of its subgroups.
Or, equivalently, describe all ergodic continuous (nonatomic) Borel probability
$\operatorname{Ad}(G)$-invariant measures on the subset $LN(G)$ (correspondingly on $LN^2(G)$) of whole space $L(G)$.
\end{problem}

We will see the different answer on the Problem 1 and Problem 2 for infinite symmetric group.

\medskip

Remark that for ergodic TNF-measure $\nu$ the zero-one law with respect to $LN(G)$ takes place:
either $\nu(LN(G))=0 \quad \mbox{or} \quad \nu(LN(G))=1$.
It is interesting to characterize the TNF actions
of a given group $G$ from the point of view of ergodic theory
and that of representation theory: what kind of
ergodic properties can have TNF actions,
and what kind of factor representations can arise for TNF actions? etc.

It is interesting also to describe other classes of  $\operatorname{Ad}(G)$-invariant measures
depending on the property of subgroup of full measure (or to describe random subgroup of the various algebraic types).

\subsection{THE CANONICAL SKEW PRODUCT, AND REDUCED (RTNF) ACTION.}

Now we consider the general actions and describe the canonical reduction which leads to a TNF actions.
First of all, we consider the factorization with respect to the iso-stable partition $\xi_G$
in order to define the \textbf{the first canonical skew product structure for general actions of groups}.

Consider a measure-preserving action of a countable group $G$ on a
Le\-besgue space $(X,\mu)$ and the $G$-invariant iso-stable
 partition $\xi_G$ (see Section~{\rm1.1}). The reduced action (quotient action) of $G$ on
 the space $(X/{\xi_G},\mu_{\xi_G})$ (see Definition~3)
 is isomorphic to the adjoint action of $G$ on the space $L(G)$ equipped with the
 characteristic measure $\Psi_*\mu$.
 We regard the triple $(L(G),\operatorname{Ad}(G),\Psi_*\mu)$  as
 the base of the canonical skew product structure for the action
 $(X,G,\mu)$.

Recall that a skew product
is defined if we have an action of the group on the base and a 1-cocycle
on the base with values in the group of automorphisms of the typical
fiber $ (Y,\alpha) $. For a free action of $G$, this skew product is
trivial, the base consists of a single point,
and the fiber is $(Y,\alpha) \equiv (X,\mu)$.  In the case of a TNF
action, $Y$ is a one-point space
and the base coincides with the whole space $(X,\mu)\equiv
(L(G),\Psi_*\mu)$.

\begin{definition}
 The action of the group $G$ on the base  $(X/{\xi_G},\mu_{\xi_G})\simeq  (L(G),\Psi_*\mu,\operatorname{Ad}(G),)$ we call canonical reduced action.
  \end{definition}
   The action on the space $(X,\mu, G)$ becomes a skew-product with fiber $(Y,\alpha)$, and we have the 1-cocycle $c$
   which is a measurable function on the product of the group and base with values in the group of measure-preserving transformations of the fiber $(Y,\alpha)$: $c:G\times L(G) \rightarrow Aut Y$.
Recall that the general condition on a cocycle $c$ on the space $L(G)$ with the
adjoint action of the group $G$ and an arbitrary group of
coefficients has the form
$$c(g_1g_2,H)=c(g_1,H)c(g_2, g_1H{g_1}^{-1}), \quad c(e,H)={\rm Id},$$
where $g\in G$, $H\in L(G)$, and $c(\cdot,\cdot)$ is a measurable function on $G\times L(G)$ with values in some group.
From this we can conclude that for a fixed $H$, the restriction of the map $g\mapsto c(g,H)$ to $g\in H$
is a homomorphism of the group $H$.
But our cocycle that defines the canonical skew product has a
stronger property.

\begin{statement}
  If the action of $G$ on $(X,\mu)$ is ergodic, then the above
  construction defines a decomposition of the space $(X,\mu)$
into the direct product $(X \approx L(G)\times Y; \mu_{\xi_G}\times \alpha) $,
where $(Y,\alpha)$ is the typical fiber of the skew product; the
action of $G$ on the base is the adjoint action with the
$\operatorname{Ad}(G)$-invariant measure
$\Psi_*(\mu)$; a 1-cocycle $c(\cdot,\cdot)$
is a function on the space $G \times L(G)$ with values in the group
$\operatorname{Aut}(Y,\nu)$
of measure-preserving automorphisms of the space $(Y,\nu)$.

A necessary and sufficient condition for a cocycle $c$ to define
the canonical skew product is as follows.

{\rm1.} If $\nu$ is a TNF measure, which means that action on the base is TNF-action, then
$$c(h,H)={\rm Id}$$
for $\nu$-almost all $H\in L(G)$ and $h \in H$,
where ${\rm Id}$ is the $\bmod 0$ identity map on the space $Y$;
in other words, the homomorphism mentioned above is identical.

{\rm2.} If $\nu$ is not a TNF measure, then, in addition to the previous condition, the following
property holds:
for $\nu$-almost all $H\in L(G)$ and $h\in N(H)\setminus H$,
$$\alpha (\operatorname{Fix}_{c(h,H)})=0,$$
where $\operatorname{Fix}_\phi=\{y\in Y:\phi(y)=y\}$ is the set of
fixed points of an automorphism $\phi$.
\end{statement}

The first condition means that  for $\nu$-almost all subgroups
$H$, for all $h\in H$, almost all points of $Y$ are fixed points
of the automorphism $c(h,H)$;
and the second condition means that for $\nu$-almost all subgroups $H$,
for all $h\notin H$,
the set of fixed points of $c(h,H)$
is, on the contrary, of zero measure.
These conditions on a cocycle follow
from that fact that, by definition, the fibers of the partition $\xi$
coincide with the sets of points with a given
stabilizer. We will not discuss details and similar questions.
Unfortunately, it is difficult to verify whether there exists a
cocycle satisfying this condition for a given $\operatorname{Ad(G)}$-invariant measure $\nu$.
At the same time, for the analysis of a given action it is important
to use the canonical skew product and to study the corresponding
$\operatorname{Ad}(G)$-invariant measure
on $L(G)$ and cocycle.

Now let us consider the action of the group $G$ on the base of the canonical skew-product.
We assume that this action is effective, in opposite case we must apply al arguments to the
quotient group of group $G$ over kernal.
\begin{definition-theorem}
Let $\pi:X\rightarrow (X/{\xi_G})$ -a canonical projection over iso-stable partition;
the stabilizer of the projection of the point with respect to reduced action is normalizer of
the  stabilizer: $$ Stab_{\pi(x)}=N(Stab_x).$$
The action of a group $G$ on a space $(X,\mu)$ is called {\it reduced
totally nonfree (RTNF)} if the reduced action is a TNF action, or,
equivalently, (see section above) $\mu\{x:N(G_x)\in LN(G)\}=1  \Leftrightarrow  \mu\{x:N^2(G_x)=N(G_x)\}=1$.
\end{definition-theorem}

The adjoint action of the element of $g\in Stab_{\pi(x)}$ must preserve the stabilizer of $Stab_x$
by construction, this means that $g \in N(Stab_x)$ and  and and vice versa.
{\it A RTNF-measure on $L(G)$ is, by definition, the characteristic measure
$\nu$ of a RTNF action} and has property: $\nu\{H\in L(G):N^2(H)=N(H)\}=1$
(the second normalizer of a subgroup coincides with its first
normalizer).

The following commutative diagram shows the first step of our classification:
$$
\begin{CD}
X @>\Psi>> L(G)\\
@V\pi VV @VV N V\\
X/\xi_G @>\Psi>> L(G).
\end{CD}
$$
Here the map $\Psi:y\mapsto G_y$ associates with a point $y$
       its stabilizer, $\pi: X \rightarrow X/\xi_G$ is the canonical projection,
and  the map     $ N:H\mapsto N(H)$ associates with a subgroup its normalizer.

In general, the quotient of the canonical skew product is not a TNF action,
because the stabilizer of a point of the quotient space is the
normalizer of the stabilizer of the original point,
$\operatorname{Stab}_{\Psi(x)}=N(\operatorname{Stab}_x)$,
but different stabilizers may have the same normalizers, and,
consequently, we can obtain different points with the same stabilizers.
Thus we can apply again the map $\Psi$ to the base $(X/{\xi_G},\nu_{\xi})$ and consider the second
canonical skew product of the base, the third one, etc. This gives an hierarchy of
$\operatorname{Ad}$-invariant measures on $L(G)$.

\begin{definition}
We will denote by $AD(G)$ the space of all
$\operatorname{Ad}(G)$-invariant measures on the lattice $L(G)$ (which will be called in short ``AD-measures'').
Denote by ${\cal N}=N_*$ the operation on the set of measures on
$L(G)$ corresponding to the {\it normalization of subgroups}:
$[{(\cal N)}\nu](F)=\nu(N^{-1}F)$, $F\subset L(G)$.
It is clear that ${\cal N}$ sends AD-measures to AD-measures.
\end{definition}

It follows from our definitions that if a measure $\nu$ on $L(G)$ is
RTNF, then the measure ${\cal N}(\nu)$ is TNF; in particular,
if $\nu$ is a TNF measure, then  ${\cal N}(\nu)=\nu$. Thus the
operation ${\cal N}:\{\mbox{RTNF measures}\}\rightarrow\{\mbox{TNF
measures}\}$ is a projection.

We have a hierarchy of AD-measures on the lattice $L(G)$:
$$AD\supset {\cal N}(AD) \supset {\cal N}^2(AD) \supset\dots\supset RTNF \supset TNF.$$

It is natural to assume that for some groups $G$, the chain of these
normalizations as well as the chain of the steps of reductions above can be infinite or even transfinite.
The most interesting classes of $AD$-measures TNF, RTNF, and AD itself.

Remark that for RTNF measure $\nu$ the measure ${\cal N}(\nu)$ is TNF measure, and, although $\nu$ is not TNF-measure,
the adjoint action of the group $G$ on the $(L(G), \nu)$ is TNF-action. Indeed, by definition of RTNF for
$\nu$-almost all subgroup $H$, $N^2(H)=N(H)$, but $N(H)$ is the stabilizer of $H$, so $\nu$-almost all stabilizers a self-normal.
Moreover, adjoint action of $G$ on the space $(L(G),\nu)$ for RTNF-measure $\nu$ is metrically isomorphic the adjoint
action of $G$ on the space  $(L(G),{\cal N}(\nu))$  and normalization $N:L(G)\rightarrow L(G)$ is that
isomorphism of the spaces and actions.

\subsection{REMARKS}
\medskip
\textbf{1. $AD$-measures}
The natural question -is it true that each ergodic $AD$-measures is characteristic measure for some
ergodic action of the group $G$. We formulate the necessary and sufficient condition on $AD$ measure
to be characteristic.
\begin{statement}
Suppose $\nu$ is ergodic $AD$-measure on $L(G)$; $\nu$ is characteristic measure for an ergodic action of $G$ on a space
$(X,\mu)$ iff there exist probability $AD$-measure $\bar \nu$ with properties:

1) adjoint action of $G$ on $(L(G),\bar\nu)$ is ergodic;

2) ${\cal N}(\bar \nu)=\nu$

 In this case we can define $X=L(G),\mu=\bar \nu$.
\end{statement}

It is not clear if such a measure $\bar \nu$ exists for all $AD$-measures $\nu$.

\medskip
\textbf{2. Fibre bundle over $L(G)$.}
Each subgroup $H$ is normal subgroup in its normalizer $N(H)$,
so we have a fibre bundle over $L(G)$ with a fibre over $H\in L(G)$---
the group $N(H)/H$. This bundle is invariant under the adjoint action of the group.
We will use it for the theory of characters of the group.

\medskip
\textbf{3. The TNF limit of the normalizations of AD-measures.}
It is natural to assume that for every measure $\nu$ from the class of
AD-measures on a given group $G$, the AD-measure $\cap \nu$ does exist,
which is the limit in $n$ of the sequence of successive normalizations:  $\nu \mapsto
{{\cal N}_*}^n(\nu)$,  $n=1,2,\ldots $. This limit (for some groups,
it may be transfinite) must be a TNF measure.

\medskip
\textbf{4. TNF actions for continuous groups.}
As we know, the $\sigma$-fields ${\mathfrak A}_G$ and ${\mathfrak
A}^G$ can be different. It is natural to define a TNF action
of a general group as an action for which the $\sigma$-field
${\mathfrak A}^G$ is the complete $\sigma$-field, or, for which
the stabilizers separate points.
In this case, we again have an isomorphism between a TNF action and the adjoint action on the lattice of subgroups.

\medskip
\textbf{5. The continuous version of combinatorial multi-transitivity.}
The continuous counterpart of the notion of transitivity (or
topological transitivity) of actions of discrete groups is that of ergodicity.
What is the analog of double transitivity? A common explanation is
that this is the ergodicity of the action
on the Cartesian square. But I believe that this parallel is too weak. The definition of double transitivity
in combinatorics can be formulated as the transitivity of the action
of the stabilizer of a point $x$ on the space $X\setminus x$.
Thus we suggest the notion of multiple transitivity, which is related
to our consideration as follows.

\begin{definition}
Assume that a countable group $G$ acts on a standard space $(X,\mu)$
with a $G$-invariant continuous measure.
We say that the action is metrically double transitive if for $\mu$-almost every
point $x\in X$, the action of the stabilizer $G_x\subset G$
on $(X,\mu)$ is ergodic. We say that the action is metrically
$k$-transitive if for almost every (in the sense of the measure $\mu^k$
on $X^k$) choice of points $x_1,x_2,\dots, x_{k-1}$, the action of the intersection of
subgroups $\bigcap_{i=1}^{k-1} G_{x_i}$
on $(X,\mu)$ is ergodic.
\end{definition}

It is natural to consider this definition only for TNF actions of $G$.
It will be clear that all TNF actions of the infinite symmetric group are $k$-transitive for an arbitrary positive integer $k$.

\medskip
\emph{It is of interest to find all countable groups for which TNF
$k$-transitive actions  exist for any positive integer $k$}.
\medskip

This problem is perhaps related to the class of {\it oligomorphic
groups}, which was defined by P.~Cameron \cite{Cameron}
(a subgroup $G$ of the group $S_{\Bbb N}$ of all finite permutations
of $\Bbb N$ is called oligomorphic if for any positive integer $k$,
the number of orbits of the diagonal action of $G$ in the Cartesian product ${\Bbb N}^k$ is finite).

\section{THE LIST OF RANDOM SUBGROUPS OF THE INFINITE SYMMETRIC GROUP}

\subsection{SIGNED PARTITIONS AND SIGNED YOUNG SUBGROUPS OF SYMMETRIC GROUPS}

We consider the countable group $S_{\Bbb N}$, the infinite symmetric
group of all finite permutations of the set of positive integers
$\Bbb N$ (or an arbitrary countable set). In this section, we will
give the list of all AD-measures on the lattice $L(S_{\Bbb N})$ of subgroups
of this group and, in particular, the list of TNF measures.
We will use some classical facts about permutation groups and the probabilistic approach.

The lattice $L(S_{\Bbb N})$ is very large and contains very different
types of subgroups. Nevertheless,
the support of an AD-measure consists of subgroups of a very
special kind: so-called signed Young groups.
The topology and the Borel structure on $L(S_{\Bbb N})$ are
defined as usual; this is a compact (Cantor) space.

\begin{definition}[Signed partitions]
 A signed partition $\eta$ of the set $\Bbb N$
 is a finite or countable partition $\Bbb N=\cup_{B\in{\cal B}}B$ of
 $\Bbb N$ together with a decomposition ${\cal
 B}={\cal B}^+\cup{\cal B}^-\cup{\cal B}^0$ of
 the set of its blocks, where ${\cal B}^0$ is the set of all
 single-point blocks; elements of ${\cal B}^+$ are called positive
 blocks, and elements of ${\cal B}^-$ are called negative blocks (thus
 each positive or negative block contains at least two points), and we
 denote by $B_0$ the union of all single-point blocks:
 $B_0=\cup_{\{x\}\in{\cal B}^0}\{x\}$.

 Denote the set of all signed partitions of $\Bbb N$ by ${\rm SPart}(\Bbb N)$.
\end{definition}

Recall that in the theory of finite symmetric groups, the Young
subgroup $Y_\eta$ corresponding to an ordinary partition
$\eta=\{B_1,B_2,\dots, B_k\}$
is $\prod_{i=1}^k S_{B_i}$,
where $S_B$ is the symmetric group acting on $B$. We will define the more general
notion of a {\it signed Young subgroup}, which makes sense both for
finite and infinite symmetric groups.
We will use the following notation: $S^+(B)$ is the symmetric group
of all finite permutations of elements of a set $B\subset \Bbb
N$, and $S^-(B)$ is the alternating
group on $B$.\footnote{Traditionally, the alternating group is denoted by $A_n$;
V.~I.~Arnold was very enthusiastic about the idea to denote it by
$S^-_n$ in order not to confuse it with the Lie algebra $A_n$; I agree
with this idea.}

\begin{definition}[Signed Young subgroups]
The signed Young subgroup $Y_{\eta}$ corresponding to a signed
partition $\eta$ of $\Bbb N$ is
$$Y_{\eta}=\prod_{B\in{\cal B}^+}S^+(B)\times \prod_{B\in{\cal B}^-}S^-(B).$$
\end{definition}

Note that  on the set $B_0\subset
\Bbb N$, the subgroup $Y_{\eta}$ act identically, so that the partition into the orbits of $Y_{\eta}$
coincides with $\eta$.
%, up to our agreement that the block $B_0$ is the
%union of single-point subsets.

It is not difficult to describe the conjugacy class of Young subgroups
with respect to the group of inner automorphisms:
$Y_{\eta}\sim Y_{\eta'}$ if and only if $\eta$ and $\eta'$ are
equivalent up to the action of $S_{\Bbb N}$. But it is more important
to consider the conjugacy
with respect to the group of outer automorphisms. This is the group
$S^{\Bbb N}$
of all permutations of $\Bbb N$.
Denote by $r_0^\pm$ the number of infinite positive
(respectively, negative) blocks, and by $r^{\pm}_s$ the number of
finite positive
(respectively, negative) blocks of length $s>1$. Obviously, the list
of numbers $\{r^{\pm}_0,r^{\pm}_1, \dots \}$ is a complete set of
invariants of the group of outer automorphisms.

\subsection{STATEMENT OF THE MAIN RESULT}

Consider a sequence of positive numbers $\alpha=\{\alpha_i\}_{i\in
\Bbb Z}$ such that
$$\alpha_i\geq \alpha_{i+1}\geq 0\mbox{ for }i>0;
\quad \alpha_{i+1} \geq \alpha_i\geq 0\mbox{ for }i<0;\quad
\alpha_0\geq 0;\quad \sum_{i\in \Bbb Z} \alpha_i= 1.$$

Consider a sequence of $\Bbb Z$-valued independent random variables $\xi_n$, $n\in
\Bbb N$, with the distribution
$$\operatorname{Prob}\{\xi_n=v \}=\alpha_v \quad \mbox{for all} \quad
n\in {\Bbb N},\, v \in \Bbb Z.$$
Thus we have defined a Bernoulli measure  $\mu_{\alpha}$ on the space of
integer sequences $$\Bbb Z^{\Bbb N}=\{\xi=\{\xi_n\}_{n\in \Bbb N}:  \xi_n \in \Bbb Z\}.$$

\begin{definition}[A random signed Young subgroup and the measures
$\nu_{\alpha}$]
Fix a sequence $\alpha=\{\alpha_i, i\in \Bbb Z\}$, and corresponding
Bernoulli measure  $\mu_{\alpha}$; for each
realization of the random sequence $\{\xi_n\}$, $n\in \Bbb N$,
  with the distribution $\mu_\alpha$, define a random signed partition
  $\eta(\xi)$ of $\Bbb N$ as follows:
$$\eta(\xi)=\{B_i\subset {\Bbb N}, i\in \Bbb Z\},\qquad
B_i:=\{n \in {\Bbb N}:\xi_n=i\},$$
here ${\cal B}^+=\{B_i,i>0\}; {\cal B}^-=\{B_i,i<0\}$,
and $B_0$ is understood as the union of one-point blocks.
The correspondence $\xi\mapsto \eta(\xi)$ defines a probability
measure on the set $\operatorname{SPart}(\Bbb N)$ of signed partitions,
or random signed partition; the image of the Bernoulli measure
$\mu_{\alpha}$. The correspondence $\xi \mapsto Y_{\eta(\xi)}$ defines
a measure, which we denote by $\nu_{\alpha}$, on the set of signed
Young subgroups, i.e., a measure on the lattice
$L(S_{\Bbb N})$ of subgroups of  $S_{\Bbb N}$.
\end{definition}

Note that all nonempty blocks of the random signed partition $\eta({\xi})$
that consist of more than one point are infinite
with $\nu_{\alpha}$-probability one.

Now we describe the list of all AD and TNF measures for the group
$S_{\Bbb N}$.

\begin{theorem}
{\rm 1.Every measure $\nu_{\alpha}$ is a Borel ergodic AD-measure on the lattice $L(S_{\Bbb N})$;
every ergodic probability Borel AD-measure on this lattice
coincides with the measure $\nu_{\alpha}$ for some $\alpha$.

\rm 2. The measure $\nu_{\alpha}$ is RTNF-measure for all alpha, and is TNF-measure if and only if
$\alpha_i=0$ for all $i\le0$. So adjoint action of the group $S_{\Bbb N}$ on the lattice $L(S_{\Bbb N})$
with any AD-measure is TNF-action.}
\end{theorem}

\subsection{PROOFS}

\begin{proof}
   1. The easy part of the proof is to check that the measures
   $\nu_{\alpha}$ are indeed ergodic AD-measures on $L(S_{\Bbb N})$.
The invariance follows from the construction, because
$\mu_\alpha$, being  a Bernoulli measure, is invariant under all permutations of indices.
The symmetric, alternating, and identity subgroups of the symmetric
groups $S(B)$ are normal, so they are $\operatorname{Ad}(G)$-invariant.
 Consequently, the measure $\nu_{\alpha}$, being the image of
 $\mu_{\alpha}$, is $\operatorname{Ad}(G)$-invariant. The ergodicity with respect to
 permutations also follows from the ergodicity of the Bernoulli measure.

  2. Now suppose that we have an ergodic AD-measure on
  $L(S_{\Bbb N})$. We will filter out, step by step, classes of
  subgroups of $S_{\Bbb N}$ that cannot support any
  AD-measure, and will finally obtain the class of signed
  Young groups as the only possible class.
Then we will construct all AD-measures on this class.

a) A classical result of the theory of permutation groups asserts that the group $S_{\Bbb N}$ has no primitive subgroups
except the whole group  $S_{\Bbb N}$ and the alternating
group.\footnote{A primitive subgroup is a subgroup that has no nontrivial invariant partitions.}
This is a more or less direct corollary of the fundamental
estimates obtained
by C.~Jordan for finite symmetric groups, which were generalized by H.~Wielandt \cite{W}
(see, e.g., \cite[Chapter~8]{DM}). Namely, this is a corollary of
Jordan's theorem asserting that if a primitive subgroup of $S_n$ has
an element with support of size $k$,
then $n<\beta(k)$; a sharp bound on $\beta(k)$ is still
unknown.

b) Now consider an AD-measure $\nu$ on the set of imprimitive but transitive subgroups of $S_{\Bbb N}$.
Assume that such a subgroup $H\in L(S_{\Bbb N})$ has an invariant partition $\theta$.
For the action of $H$ to be transitive, all nontrivial minimal blocks
of $\theta$ must have the same length $l>1$, which must be finite
(because finite permutations cannot move one infinite block to
another one).\footnote{Note that the lengths of all blocks for a given imprimitive group
may be either bounded (so-called ``almost primitive groups'') or
unbounded (``totally imprimitive subgroups''), see \cite{DM}, but this
difference is not important for our purposes.} Denote by $\theta(H)$
the partition of $\Bbb N$ into the minimal blocks for
$H$. The map $H \rightarrow \theta(H)$ associates with $\nu$-almost
every imprimitive subgroup a partition into blocks of length $l(H)$;
because of the $\operatorname{Ad}(G)$-ergodicity of $\nu$, this length is the same for
$\nu$-almost all subgroups $H$. Thus the map $H \mapsto \theta(H)$
sends $\nu$ to a probability measure on the set of
partitions of $\Bbb N$ with countably many blocks of the same length
$l>1$, and this measure is invariant with respect to the action of
$S_{\Bbb N}$ on the space of such partitions. Let us show that there are no such finite measures.

\begin{lemma}
There are no probability measures on the space
$\operatorname{Part}(l)$ of all partitions of $\Bbb N$ into (countably
many) blocks of length $l>1$ that are invariant
with respect to the group $S_{\Bbb N}$.
\end{lemma}

\medskip\noindent{\bf Remark.} The space $\operatorname{Part}(l)$
equipped with the weak topology is locally compact but not compact;
its natural compactification consists of
all partitions whose blocks have length at most $l$.

\begin{proof}
Consider the case $l=2$, the same proof is true for an arbitrary $l$.
Each partition from $\operatorname{Part}(2)$ determines a
symmetric matrix (for $l>2$ -symmetric tensor) $\{a_{i,j}\}$, $a_{i,j}=a_{j,i}$, $a_{i,i}=0$,
$i,j=1,2, \dots$, with only one entry in each row and each column equal to $1$, all the
other entries being equal to $0$.
But because of the $S_{\Bbb N}$-invariance, we have a random symmetric matrix $\{a_{i,j}\}$,
unique element in each row which is equal to 1 must be uniformly distributed along its row.
 It is impossible for infinite matrix.
\end{proof}

c) We have proved that an AD-measure on the lattice
$L(S_{\Bbb N})$ takes
the value $0$ on the set of all transitive subgroups; so we have reduced the analysis to
intransitive subgroups. Fix such a generic intransitive subgroup $H\in
L(S_{\Bbb N})$ and consider the maximal partition $\eta(H)$ into its
transitive components. The action of the group $H$ on each such
component must be primitive, because imprimitive cases can be
discarded for the same reason as in the previous part of the proof.
For the same reason, it is obvious that all components of this
partition must be infinite. Consequently, the action of $H$ on each
component is either the action of the whole symmetric group, or that
of the alternating group (see part~a) of the proof), or that of the
identity group on the single-point blocks. We denote the blocks by
$B_i$, $i>0$, when the action of $H$ is the action of the symmetric
group of $B_i$, and by $B_i$, $i<0$, when the action of $H$ is the
action of the alternating group of $B_i$.
 The action of the identity group on all single-point blocks can be
 regarded as the identity action on the union of such blocks $B_0$.
 Thus we have a {\it signed partition} $\eta(H)$ such that the action
 of $H$ on each block $B_i$, $i>0$, is the action of $S^+(B)$, the action of $H$ on
 each block $B_i$, $i<0$, is the action of $S^-(B)$,
 and the action on $B_0$ is the identity action. This means that
 $H\subset Y_{\eta} =\prod_i S^{\pm}(B_i)$ and the restriction of the
 action of $H$ to $B_i$ is the action of $S^{\pm}(B)$.

 For each $i\ne 0$ denote the group
 $$K_i =\{g\in S^{\pm}(B_i):\exists {\bar g}\in H,\quad {\bar g}|_{B_j}=id, \quad \forall j\ne i, \quad {\bar g}|_{B_i}=g \} $$
 or a subgroup of the all elements in $H$ which acts as identity on all $B_j, j\ne i$.

 It is clear that $K_i$ is a normal subgroup  in $S^{\pm}(B_i)$ (because it is the kernal of homomorphism),
 so $K_i$ is either $S^{\pm}(B_i)$ or $K_i=\{id\}$, and $$\prod_{i\ne 0} K_i \subset H$$
 Thus we need to prove that $K_i=H|_{B_i}=S^{\pm}(B_i)$ for all $i\ne 0$ (and in particular $K_i\ne \{id\} \quad \mbox{if} \quad i\ne 0$).
 There are no problem with $i$ if $K_i=S^+(B_i)=H|_{B_i}$. We must consider two cases: the first case when $K_i=\{id\}$ but $H|_{B_i}=S^{\pm}(B_i)$ (in this case it does not matter $H|_{B_i}=S^+$ or $S^-$, so $i\ne 0$), and the second  case when $K_i= S^-(B_i)\ne H|_{B_i}=S^+(B_i)$.
 Let us consider the first case. Suppose for some $i\ne 0 \quad K_i=\{id\}$ but $H|_{B_i}=S^{\pm}(B_i)$. Then there exist at least one $j\ne i$
  for which  $$K_i =\{g\in S^{\pm}(B_i):\exists {\bar g}\in H,\quad {\bar g}|_{B_j}=id,  \quad {\bar g}|_{B_i}=g \},$$  indeed the intersection could be either $\{id\}$ or $S^{\pm}(B_i)$ for all $j$ and if the intersection in the definition of $K_i$ over all $j\ne i$ is $\{id\}= K_i$, then
  such $j$ exists. It means that for this $j$  and  for $h\ne id, h \in  H|_{B_i}$ there exists $h'\in H|_{B_j}$ and $g\in H$
such that $g|_{B_i}=h, g|_{B_j}=h'$. So we have a map from $H|_{B_i}$ to $H|_{B_j}$ which is homomorphism, and consequently
isomorphism which is simply bijection - $T$ - between $B_i$ and $B_j$. This bijection could be arbitrary because of invariance under conjugation
of the group. Thus the action on of the group $H$ on $B_i\cup B_j$ is as follow: if $n\in B_i$ and $Tn = m \in B_i$, then $gm=Tgn$.
or $gT=Tg$ on $B_i\cup B_j$. If we restrict the action of $H$ on  $B_i\cup B_j$ only, we obtain that the group $H$ acts periodically
(or "simultaneously") on $B_i$ and $B_j$.

\begin{lemma}
There are no AD-invariant measures which are concentrated on the intransitive subgroups $H \subset S_{\Bbb N}$ of the following type:
If ${\Bbb N}={\Bbb N'}\times K$, ($N'$ is infinite), then
$$H=S_{\Bbb N'}\times \{id_K\}\subset S_{\Bbb N} \quad  H=\{g: g=(g',id_K); g'\in S_{\Bbb N'}\},$$
or periodic action on ${\Bbb N'}\times K$.
 \end{lemma}
 \begin{proof}
 The random group $H$ of this type must define a $S_{\Bbb N}$-invariant random partition of $\Bbb N$ onto $|K|$ parts
 and $S_{\Bbb N'}$- invariant random bijections between all parts. The invariant random partitions do exist -see the next item but
 invariant bijection do not because the absence of probability measure on the group $S_{\Bbb N'}$.
  \end{proof}
So we don't need to consider the subgroups $H$ for which the first case takes place and consequently we already proved
that $K_i=S^{\pm}(B_i) \quad i\ne 0$ (we write $S^{\pm}$ when it is not important either $S^+$ or $S^-$).

Suppose now that for some $i$, $K_i\supset S^-(B_i)$, and $H|_{B_i} = S^+(B_i)$.
Again find $j\ne 0$ for which  $K_i =\{g\in S^+(B_i):\exists {\bar g}\in H,\quad {\bar g}|_{B_j}=id, \quad {\bar g}|_{B_i}=g \} $.
 Because of definition of $K_i$ it is clear that $H|_{B_i\cup B_j}\supset S^-(B_i)\times S^{\pm}(B_j)$, and the last subgroup
 has index in $H|_{B_i\cup B_j}$ at most two, but also we have $H|_{B_i}=S^+(B_i)$, so $H|_{B_i\cup B_j}=S^+(B_i)\times S^{\pm}(B_j)$.
 But this means that $K_i=H|_{B_i}=S^+(B_i)$.

So we prove that $H = \prod_i H|_{B_i}$ and each $H|_{B_i}=S^{\pm}(B_i)$ for $i\ne 0$, or in another words we have proved that
$H$ must be a signed Young subgroups: only signed Young subgroups can carry AD-invariant measures
on the lattice $L(S_{\Bbb N})$.

The measures $\nu{\alpha}$ which was defined above are concentrated on the signed Young subgroup by definition.

 d)Now we will prove that indeed this case is realized: the random subgroups in the infinite symmetric group
or AD-invariant ergodic measure on $L(S_{\Bbb N}$ is one of the measure $\nu_{\alpha}$ and indeed each measure $\nu_{\alpha}$
are AD-invariant ergodic measure on $L(S_{\Bbb N}$.

  We must identify the required measures with the ergodic limits
 with respect to conjugation of signed Young subgroups. Because of the
 correspondence between signed Young subgroups and signed partitions,
 this question is equivalent to the description of $S_{\Bbb
 N}$-invariant measures on the set of signed partitions. The last
 question is similar to the classical de Finetti problem concerning
 $S_{\Bbb N}$-invariant measures on the space of all functions on
 $\Bbb N$ (see \cite{V74}). The only small difference lies in the fact
 that, in contrast to the classical situation, here we have three
 types of blocks of signed partitions instead of one type in the ordinary de Finetti theorem.

 \begin{lemma}[An analog of classical de Finetti's theorem; Kingman's theorem \cite{Ki}]
 Consider the space $\operatorname{SPart}(\Bbb N)$ of signed
 partitions of $\Bbb N$; every ergodic $S_{\Bbb N}$-invariant measure
 on  $\operatorname{SPart}(\Bbb N)$ is determined
 by a sequence
 $\alpha=\{\alpha_i\}_{i\in \Bbb Z}$ such that $\alpha_i\geq \alpha_{i+1}\geq 0$
 for $i>0$, $\alpha_{i+1} \geq \alpha_i\geq 0$
for $i<0$, $\alpha_0\geq 0$, and $\sum_{i\in \Bbb Z} \alpha_i= 1$,
as described above.
 \end{lemma}

\begin{proof}
The lemma can be proved by any of the methods people use to prove
de Finetti's theorem. For completeness, we present a proof,
applying  our old ergodic method from \cite{V74}. In order to find all
ergodic measures $\nu$ on a
compact $S_{\Bbb N}$-space $X$ using the pointwise
ergodic theorem for the group $S_{\Bbb N}$ (which is in fact a theorem
on the convergence of martingales), it suffices to find the weak
limits of measures (when they do exist)
$$\lim \frac{1}{n!}\sum_{g \in S_n} \delta_{gx}$$
for all $x \in X$. More exactly, we need to calculate the limits
  $$\lim_n \frac{1}{n!}\sum_{g \in S_n} f(gx)$$
for continuous functions $f \in C(X)$. In our case, it suffices to
consider cylinder functions on $\operatorname{SPart}(\Bbb N)$ which depend on finitely
many blocks. Fix a signed partition $\eta$ and
label its blocks with integers in an arbitrary way so that positive
(negative)
integers correspond to positive (negative) blocks and $B_0$ is the
union of one-point blocks.
Consider the $\Bbb Z$-valued sequence $x_n$, $n \in \Bbb N$, defined
as follows:
$x_n=s$ if $n\in B_s$. Now we may say that the signed partition $\eta$
is the partition corresponding to the sequence $\{x_n\}$,
and each such sequence determines a signed partition. The action of
$S_{\Bbb N}$ on the set of signed partitions and its action by
permutations of coordinates of sequences obviously agree, so our
problem reduces to the description of $S_{\Bbb N}$-invariant measures on the space of all
elements of ${\Bbb Z}^{\Bbb N}$. But this is exactly de Finetti's
problem. Start with an arbitrary sequence $\{x_n\} \in {\Bbb Z}^{\Bbb N}$
and calculate the limit
$$\lim_n\frac{1}{n!}\#\{g\in S_n: x_{gn}=v\}=\alpha_v$$
under the assumption that it does exist (it exists for almost all $x \in X$).
Thus we obtain the one-dimensional distribution of the random (with
respect to the measure $\nu$)
sequence $x_n$.
In order to prove that this measure is a Bernoulli measure on ${\Bbb
Z}^{\Bbb N}$, we must calculate the joint distribution of several
coordinates of $x_n$. But because of the complete transitivity of the
action of $S_n$, for any choice of $v_1,\dots, v_t$ and for $n\gg t$
we have
$$\frac{\#\{g\in S_n: x_{gi}=v_i,\;i=1,2,\dots, t,\, n\gg t\}}{n!} \approx \prod_{i=1}^t \alpha_{v_i}, $$
which means that the random variables $x_n$, $n=1,2, \dots$, are
independent. Thus all AD-measures arise from Bernoulli measures
on the space of signed partitions, i.e., $\nu=\nu_{\alpha}$ for some $\alpha$.
The assertion of the theorem for the un-signed partitions is Kingman's theorem (\cite{Ki}), but
our proof is different.
\end{proof}

e) Consider the random signed Young subgroup $Y_\eta$
constructed from a sequence $\alpha$ with $\alpha_i=0$ for all $i\leq 0$. Then all blocks
$B_i$, $i\leq 0$, are empty with
probability one. Then, obviously,
the normalizer $N(Y_{\eta})$ coincides with $Y_{\eta}$, since each
block of $\eta$ gives rise to the self-normalizing subgroup $S^+(B_i)$.
Consequently, the measure $\nu_{\alpha}$ is TNF. If $\alpha_i>0$ for some
$i\leq 0$, then the corresponding block $B_i$ is not empty with
probability one, whence $N(S^-_{B_i})=S^+_{B_i}\ne S^-_{B_i}$, so
that $Y_{\eta}$ is not self-normalizing. But
$$N(Y_\eta)=\prod_{i=-\infty}^{+\infty}S_{B_i}^+.$$ Thus
$N^2(Y_{\eta})=N(Y_{\eta})$,
so that $\nu_{\alpha}$ is a RTNF measure in the terminology of
Section~1. This completes the proof of Theorem~1.
\end{proof}

\begin{corollary}
The action of the group $S_{\Bbb N}$ on the measure space $(L(S_{\Bbb N}),\nu_{\alpha})$ is ergodic.
\end{corollary}

Indeed, this is a corollary of the fact that the measure
$\nu_{\alpha}$ is the image of the Bernoulli measure $\mu_{\alpha}$
and the correspondence $\mu_{\alpha}\mapsto\nu_\alpha$
between measures commutes with the action of the
group.
The corresponding representation of the group $S_{\Bbb N}$ in the
space $L^2_{\nu_{\alpha}}(L(S_{\Bbb N}))$
will be considered elsewhere.

\begin{corollary}
There are three degenerate measures $\nu_{\alpha}$, in the following
cases (in the parentheses we indicate the corresponding character and
representation, see below):

{\rm 1)} $\alpha_1=1$, $\alpha_i=0$, $i\ne 1$; in this case, $\nu_{\alpha}=\delta_{S_{\Bbb N}}$
($\chi(g)\equiv 1$, the identity representation);

{\rm 2)} $ \alpha_{-1}=1$, $\alpha_i=0$, $i\ne -1$; in this case,  $\nu_{\alpha}=\delta_{S^-_{\Bbb N}}$
($\chi(g)=(-1)^{\operatorname{sgn}(g)}$, the alternating
representation);

{\rm 3)} $\alpha_0=1$, $\alpha_i=0$, $i\ne 0$; in this case,
$\nu_{\alpha}=\delta_{{\rm Id}_{S_{\Bbb N}}}$
($\chi(g)=\delta_e(g)$, the regular representation).

An ergodic AD-measure $\nu_{\alpha}$ is atomic only in these
three cases (in which it is if fact a $\delta$-measure);
 in all the other cases, $\nu_{\alpha}$ is a continuous measure.
\end{corollary}

\subsection{REMARKS AND A FORMULA FOR THE MEASURES OF THE SETS OF FIXED POINTS}
\smallskip
    Make sense to compare the language which we use here (the action on $L(G)$)
 with that which was used in \cite{VK} (the action on the Bernoulli sequences).

 More concretely, consider the action of $S_{\Bbb N}$ on the space ${\Bbb Z}^{\Bbb N}$
 (instead of $L(S_{\Bbb N})$) and ask for a description of TNF and RTNF measures. The answer is a
 little bit different than for the action
 on the space of Young subgroups. Namely, the following result holds.

\begin{statement}
The measure $\mu_{\alpha}$ on the space $X={\Bbb Z}^{\Bbb N}$
with the action of the group $S_{\Bbb N}$
is a TNF measure if and only if all $\alpha_i$, $i\ne 0$, are
distinct. If $\alpha_i=\alpha_j$
for some $i\ne j$, then the action of $S_{\Bbb N}$ is RTNF but not TNF.
The canonical projection $X \rightarrow X/{\xi_G}$ is the
factorization with respect to the following equivalence relation
on  $X={\Bbb Z}^{\Bbb N}$:
two elements $\{x_n\}_{n\in\Bbb N},\{x'_n\}_{n\in
\Bbb N}\in X$ are equivalent for if for every $v\in\Bbb Z$ either
$\{n \in {\Bbb N}: x_n=v\}=\{n\in {\Bbb N}: x'_n=v\}$, or
there exists $v'\in \Bbb Z$
with $\alpha_v=\alpha_{v'}$ such that
$$\{n \in {\Bbb N}: x_n=v\}=\{n\in {\Bbb N}: x'_n=v'\},$$
and
$$\{n \in {\Bbb N}: x_n=v'\}=\{n\in {\Bbb N}: x'_n=v\}.$$
\end{statement}
 Thus, in this case the action is RTNF not TNF if we have multiplicity in the values
 of $\alpha$: $\alpha_i=\alpha_j, i\ne j$; this is not the case for the action in $L(S_{\Bbb N})$.

The supports of the measures $\nu_{\alpha}$ in the topological sense
(i.e., the minimal closed subsets of full measure) coincide for all $\alpha$
that have the same number of infinite blocks. The support
of $\nu_{\alpha}$ for $\alpha$ having infinitely many infinite blocks
coincides with the space of all signed Young subgroups.

In the case of the infinite symmetric group, all AD-measures are invariant under the group $S^{\Bbb N}$ of all permutations
of $\Bbb N$. The conjugation with respect to this group is an
extension of the usual conjugation; but for a generic subgroup
$H$ from a set $A$ of full $\nu_{\alpha}$-measure, its orbit under the action of
$S^{\Bbb N}$ is much larger than $A$.
In other words, the frequencies $\alpha_i$ are invariant under the
usual conjugation, but not under its extension. This fact is related
to the so-called
Kolmogorov effect (see \cite{kol}).

Theorem 1 gives more than just the list of AD-measures on the
group $S_{\Bbb N}$; it helps to give a new proof of
Thoma's formula for indecomposable characters of this group.
This will be the subject of our next article,
and now we merely carry out the calculations and give a short commentary.
Here we present the formula for characters in the ``positive''
case.\footnote{For convenience, we have sightly changed the
notation: usually, $\alpha_i\equiv \beta_i$ for $i<0$, and $\alpha_0\equiv\gamma$.}

 \begin{theorem}
 For an ergodic AD-measure $\nu_{\alpha}$,
$$\nu_{\alpha}(F_g)\equiv \nu_{\alpha}\{H:gHg^{-1}=H\}=\nu_{\alpha}\{H:g\in N(H)\}=\prod_{n>1} [p_n(\alpha)]^{c_n(g)}, $$
where $$p_n(\alpha)=\sum_{i\ne 0} \alpha_i^n$$ (Newton's power sum)
and $c_n(g)$ is the number of cycles of length $n>1$ of a permutation $g$.
\end{theorem}

In the case where $\alpha_i=0$ for $i < 0$, this formula coincides
with Thoma's formula \cite{VK1} for characters of the infinite symmetric group,
because the measure of the set of fixed points is equal to the value
of the character: $$\chi_{\alpha}(g)= \nu_{\alpha}(\operatorname{Fix}(g)).$$ In
the general case, Thoma's formula involves super-Newton instead of
Newton sums:
$$p_n(\alpha)=\sum_{i>0}\alpha_i^n+(-1)^{n-1}\sum_{i<0}\alpha_i^n=
\sum_{i\ne 0}(\operatorname{sgn} i)^{n-1}\alpha_i^n.$$

The measure of the set of fixed points does not depend on the types of
blocks, but for a general parameter $\alpha$, the value of the character
is not equal just to the measure of this set, the
formula involving a certain multiplier (see \cite{VK}).
 We will return to this question and give a model of representations in the next article.

\end{document}